\documentclass[final,times,1p,reqno,11pt]{elsarticle}
\usepackage[latin1]{inputenc}
\usepackage{enumerate}
\usepackage{amsfonts}
\usepackage{amsmath}
\usepackage{amsthm}
\usepackage{amsthm}
\usepackage{amssymb}
\usepackage{latexsym}
\usepackage{multicol}
\usepackage{verbatim}
\usepackage{amscd}
\usepackage{mathrsfs}
\usepackage[usenames]{color}
\usepackage{hyperref}
\usepackage{pb-diagram}

\allowdisplaybreaks[2]

\theoremstyle{definition}
\newtheorem{thm}{Theorem}[section]
\newtheorem{prop}[thm]{Proposition}
\newtheorem{cor}[thm]{Corollary}

\newtheorem{lem}[thm]{Lemma}

\newtheorem{ex}[thm]{Example}
\newtheorem{rem}[thm]{Remark}
\newtheorem{exe}[thm]{Example}

\numberwithin{equation}{section}

\let\bwdg\bigwedge
\def\bigwedge{{\textstyle\bwdg}}

\makeatletter
\def\ps@pprintTitle{%
     \let\@oddhead\@empty
     \let\@evenhead\@empty
     \def\@oddfoot{\footnotesize\itshape%
       \hfill}%\today}%
     \let\@evenfoot\@oddfoot}
\makeatother
\journal{\relax}

\newcommand{\wt}{\operatorname{wt}}

\newcommand{\nc}{\newcommand}

\renewcommand{\Bbb}{\mathbb}
\nc\bomega{{\mbox{\boldmath $\omega$}}} \nc\bpsi{{\mbox{\boldmath
$\Psi$}}}
 \nc\balpha{{\mbox{\boldmath $\alpha$}}}
 \nc\bpi{{\mbox{\boldmath $\pi$}}}

\newcommand{\lie}[1]{\mathfrak{#1}}



 \nc{\Hom}{\operatorname{Hom}}
\nc{\End}{\operatorname{End}} \nc{\wh}[1]{\widehat{#1}}
\nc{\Ext}{\operatorname{Ext}} \nc{\ch}{\operatorname{ch}} \nc{\gch}{\operatorname{gch}}
\nc{\ev}{\operatorname{ev}} \nc{\Ob}{\operatorname{Ob}}
\nc{\soc}{\operatorname{soc}} \nc{\rad}{\operatorname{rad}}
\nc{\head}{\operatorname{head}}

 \nc{\Cal}{\mathcal} \nc{\Xp}[1]{X^+(#1)} \nc{\Xm}[1]{X^-(#1)}
\nc{\on}{\operatorname} \nc{\Z}{{\mathbf Z}} \nc{\J}{{\mathcal J}}
\nc{\C}{{\mathbf C}} \nc{\Q}{{\mathbf Q}}

\nc{\N}{{\Bbb N}} \nc\boa{\mathbf a} \nc\bob{\mathbf b} \nc\boc{\mathbf c}
\nc\bod{\mathbf d} \nc\boe{\mathbf e} \nc\bof{\mathbf f} \nc\bog{\mathbf g}
\nc\boh{\mathbf h} \nc\boi{\mathbf i} \nc\boj{\mathbf j} \nc\bok{\mathbf k}
\nc\bol{\mathbf l} \nc\bom{\mathbf m} \nc\bon{\mathbf n} \nc\boo{\mathbf o}
\nc\bop{\mathbf p} \nc\boq{\mathbf q} \nc\bor{\mathbf r} \nc\bos{\mathbf s}
\nc\bou{\mathbf u} \nc\bov{\mathbf v} \nc\bow{\mathbf w} \nc\boz{\mathbf z}
\nc\boy{\mathbf y} \nc\ba{\mathbf A} \nc\bb{\mathbf B} \nc\bc{\mathbf C}
\nc\bd{\mathbf D} \nc\be{\mathbf E} \nc\bg{\mathbf G} \nc\bh{\mathbf H}
\nc\bi{\mathbf I} \nc\bj{\mathbf J} \nc\bk{\mathbf K} \nc\bl{\mathbf L}
\nc\bm{\mathbf M} \nc\bn{\mathbf N} \nc\bo{\mathbf O} \nc\bp{\mathbf P}
\nc\bq{\mathbf Q} \nc\br{\mathbf R} \nc\bs{\mathbf S} \nc\bt{\mathbf T}
\nc\bu{\mathbf U} \nc\bv{\mathbf V} \nc\bw{\mathbf W} \nc\bz{\mathbf Z}
\nc\bx{\mathbf x}

\def\opl_#1^#2{\text{\scriptsize$\bigoplus\limits_{\text{\footnotesize$#1$}}^{\text{\footnotesize$#2$}}$}}
\def\otm_#1^#2{\text{\scriptsize$\bigotimes\limits_{\text{\footnotesize$#1$}}^{\text{\footnotesize$#2$}}$}}
\def\gb#1{\text{$\boldsymbol #1$}}

\def\tsum_#1^#2{\text{\small$\sum\limits_{\text{\footnotesize$#1$}}^{\text{\footnotesize$#2$}}$}}

\def\cel#1{\mathcal #1^\ell}

\begin{document}
\begin{frontmatter}
\title{On Multigraded Generalizations of Kirillov-Reshetikhin Modules  \tnoteref{fn1}}
\author[ime]{Angelo Bianchi}
\address[ime]{Department of Mathematics, University of Campinas, SP, Brazil, 13083-970}
\emailauthor{angelo@ime.unicamp.br}{A.~B.}
\author[ucr]{Vyjayanthi Chari}
\address[ucr]{Department of Mathematics, University of California, Riverside, CA 92521}
\emailauthor{vyjayanthi.chari@ucr.edu}{V.~C.}
\author[kln]{Ghislain Fourier}
\address[kln]{Mathematisches Institut, Universit\"at zu K\"oln, Germany}
\emailauthor{gfourier@mi.uni-koeln.de}{G.~F.}
\author[ime]{Adriano Moura}\emailauthor{aamoura@ime.unicamp.br}{A.~M.}
\tnotetext[fn1]{Partially supported by the FAPESP grant 2011/22322-4 (A.~B.), the NSF grant DMS-0901253 (V.~C.), the DFG priority program 1388 - ``Representation Theory'' (G. F.), and the CNPq grant 306678/2008-0 (A.~M.)}

\begin{abstract}
We study the category of $\mathbb Z^\ell$-graded modules with finite-dimensional graded pieces for certain $\mathbb Z^\ell_+$-graded Lie algebras. We also consider certain Serre subcategories with finitely many isomorphism classes of simple objects. We construct projective resolutions for the simple modules in these categories and compute the Ext groups between simple modules. We show that the projective covers of the simple modules in these Serre subcategories can be regarded as multigraded generalizations of Kirillov-Reshetikhin modules and give a recursive formula for computing their graded characters.
\end{abstract}

\end{frontmatter}

\section*{Introduction}

The study of the structure of the simple finite-dimensional representations of quantum affine algebras has attracted a lot of attention over the past
two decades. Of particular interest is the subclass of minimal affinizations introduced in \cite{cha:minr2} which includes the class of Kirillov-Reshetikhin modules. The Kirillov-Reshetikhin modules have been studied extensively due to their application to mathematical physics and their rich combinatorial structure (cf. \cite{frke:combkr,fos:krcne,jap:path,jap:rem,her:KR,nasa:perfectkr,nanak:trc,nanak:trd,sch:combkr,scti:demkrc} and references therein).
One way of studying their structure is by looking at the classical limits which can then be regarded as graded modules for the current algebra $\lie g[t]=\lie g\otimes \mathbb C[t]$, where $\lie g$ is the underlying finite-dimensional semisimple Lie algebra. The limit process does not change the character of the representation and, hence, it can be studied by investigating the limit module. This fact is one of the key motivations for the increasing interest in the study of categories of graded modules for current algebras \cite{cg:minp,ckr:faces,foli:wdkr,kona:loewy,naoi:fuskr}.

If $\lie g$ is of classical type, then the classical limits of minimal affinizations factor to representations for the finite-dimensional algebra $\lie g[t]/(t^2\mathbb C[t])$ which is isomorphic to the semidirect product of $\lie g$ with its adjoint representation.
This motivates the study of graded modules for algebras of the form $\lie a=\lie g\ltimes V$ where $V$ is a finite-dimensional $\lie g$-module. The $\mathbb Z$-grading on $\lie a$ is set to be $\lie a[0]=\lie g$ and $\lie a[1]=V$. Consider the category $\mathscr G$ of graded $\lie a$-modules with finite-dimensional graded pieces. It was shown in \cite{cg:minp} that the isomorphism classes of simple modules of this category are in bijection with the set $\Lambda = P^+\times \mathbb Z$ where $P^+$ is the set of dominant weights of $\lie g$. Moreover, it was also shown that every simple module has a projective cover and a projective resolution in this category. This was then used to compute the Ext groups between simple modules. The connection with minimal affinizations and Kirillov-Reshetikhin modules is made by considering Serre subcategories of this category generated by certain finite sets of simple modules. More precisely, one chooses a finite subset $\Gamma$ of $\Lambda$ which is convex with respect to a certain natural partial order on $\Lambda$. Let $\mathscr G[\Gamma]$ be the corresponding Serre subcategory. Given $(\lambda,r)\in\Gamma$, let $P(\lambda,r)$ be the projective cover of the associated simple module. It was shown in \cite{cg:minp} that $P(\lambda,r)$ projects onto a module $P(\lambda,r)^\Gamma$ which is projective in $\mathscr G[\Gamma]$ (in fact, the whole projective resolution of the underlying simple module induces a projective resolution in  $\mathscr G[\Gamma]$). It turns out that one can describe the modules $P(\lambda,r)^\Gamma$ in terms of generators and relations which, in several cases, coincides with the relations satisfied by the highest-weight vectors of classical limits of minimal affinizations (cf. \cite{cm:kr,cm:krg,mou:reslim,mp:e6}). Moreover, the results on the Ext groups give rise to a recursive character formula for the modules $P(\lambda,r)^\Gamma$. Therefore, a character formula for minimal affinizations is also obtained in the cases that their classical limits are isomorphic to modules of the form $P(\lambda,r)^\Gamma$.

In the last years, the interest on the representation theory of ``generalized current algebras'', in particular the multivariable current algebras, has increased significantly \cite{bapa,cfk,kod:ext,lau:ml,ns,nss}. The study of characters in the multigraded context is well-known to be significantly more complicated than in the single variable case (see \cite{felo:mweyl} for the case of Weyl modules) and technical issues that, in the single variable setting arise when working with exceptional algebras, in the multivariable case arise for algebras of classical type. In the present paper we are concerned with the study of Kirillov-Reshetikhin modules in the multivariable setting. The main achievement is the extension to this setting of all the results described in the previous paragraph for algebras of the form $\lie a:=(\lie g\otimes\mathbb C[t_1,\dots,t_\ell])/\lie a_2$ where $\lie a_2$ is the ideal generated by all elements of degree $2$.

Following \cite{ckr:faces}, we start the study by considering a more general context. Namely, we consider the category $\mathscr G$ of $\mathbb Z^\ell$-graded modules with finite-dimensional graded pieces for a $\mathbb Z_+^\ell$-graded Lie algebra $\lie a$ such that $\lie a[\gb 0]=\lie g$ and $\lie a[\gb r]$  is finite-dimensional for all $\gb r\in\mathbb Z_+^\ell$. We show that the isomorphism classes of simple modules of this category are indexed by $\Lambda:=P^+\times\mathbb Z^\ell$ and construct a projective resolution $P_j(\lambda,\gb r), j\ge 0$, for the simple module $V(\lambda,\gb r)$ associated to $(\lambda,\gb r)\in \Lambda$. We end the first section by considering the Serre subcategories $\mathscr G[\Gamma]$ for subsets $\Gamma$ of $\Lambda$. We prove that, if $\Gamma$ is finite and convex with respect to a certain partial order on $\Lambda$ and $(\lambda,\gb r)\in\Gamma$, then the modules $P_j(\lambda,\gb r)$ projects onto finite-dimensional modules $P_j(\lambda,\gb r)^\Gamma$ giving rise to a (finite) projective resolution of $V(\lambda,\gb r)$ in $\mathscr G[\Gamma]$.

We start Section 2  assuming that $\lie a[\gb r]=0$ for all $\gb r=(r_1,\dots,r_\ell)$ such that $r_1+\cdots+r_\ell>1$ and compute all the Ext groups between simple modules in this case. Moreover, we also show that if $\Gamma$ is finite and convex and $(\lambda,\gb r), (\mu,\gb s)\in\Gamma$, then
$$\Ext^j_{\mathcal G[\Gamma]}(V(\lambda,\gb r), V(\mu,\gb s))\cong \Ext^j_{\mathcal G}(V(\lambda,\gb r), V(\mu,\gb s)) \qquad\text{for all}\qquad j\ge 0.$$
Next, we obtain a multigraded version of the recursive character formula for the modules $P(\lambda,\gb r)^\Gamma:=P_0(\lambda,\gb r)^\Gamma$ assuming that $\Gamma$ is convex with respect to a refinement of the partial order on $\Lambda$ similar to the refinement considered in \cite{ckr:faces} for the case $\ell=1$. We then obtain generators and relations for the modules $P(\lambda,\gb r)^\Gamma$ (this generalizes the original result of \cite[Theorem 1]{cg:minp} even in the case $\ell=1$). In the last subsection we work with the algebras $\lie a:=\lie g\otimes\mathbb C[t_1,\dots,t_\ell]$ and $\lie b:=\lie a/\lie a_2$ with $\lie g$ of classical type. For each $(\lambda,\gb r)\in\Lambda$, we consider the $\lie a$-module $N(\lambda,r)$ defined by a generator satisfying the naive multigraded generalization of the relations satisfied by the highest-weight vector of the classical limits of minimal affinizations computed in \cite{mou:reslim}. In particular, if $\lambda$ is a multiple of a fundamental weight, the relations simplify to the naive multigraded generalization of those considered in \cite{cm:kr,cm:krg} for Kirillov-Reshetikhin modules. We show that $N(\lambda,\gb r)$ factors to a $\lie b$-module and, as such, it is isomorphic to a module of the form $P(\lambda,\gb r)^\Gamma$ for an appropriate choice of $\Gamma$. This allows us to use our recursive character formula for computing the graded character of $N(\lambda,\gb r)$. We end the paper by giving an example of such computation. In particular, the example shows that the set of simple factors on a Jordan-H\"older series of these modules depend on $\ell$, differently to what happens with the corresponding local Weyl modules studied in \cite{cfk} (see Remark \ref{r:depell}).

\vskip15pt
\noindent{\bf Acknowledgements:}  A. Bianchi thanks the University of California at Riverside for hospitality and A. Khare for helpful discussions.

\section{Multigraded Modules and Projective Resolutions}

In this section we define the categories which are our objects of study, classify their simple objects, and construct a projective resolution for each of the simple objects.

\subsection{The algebras}

Let $\mathbb C$ be the field of complex numbers and $\mathbb Z,\mathbb Z_+$ the subsets of integers and nonnegative integers, respectively. Given $\ell>0$,  let $\gb e_1,\dots,\gb e_\ell$ be the canonical basis of $\mathbb Z^\ell$ and $\deg:\mathbb Z^\ell\to\mathbb Z$ be the group homomorphism defined by $\deg(\gb e_j )=1$ for all $j=1,\dots,\ell$. Let also $\cel V$ be the category of $\mathbb Z^\ell$-graded $\mathbb C$-vector spaces. Thus, the morphisms in $\cel V$ are linear maps $f:V\to W$ such that $f(V[\gb r])\subseteq W[\gb r]$ for every $\gb r\in\mathbb Z^\ell$. We will also use the functor $\mathscr D:\cel V\to\mathcal V^1$ given by
$$\mathscr D (V) = \opl_{r\in \mathbb Z}^{} V[r] \qquad \text{where} \qquad  V[r]=\opl_{\substack{\gb r\in\mathbb Z^\ell \\\deg(\gb r)=r }}^{}V[\gb r].$$

Fix a $\mathbb Z_+^\ell$-graded Lie algebra $\lie a=\opl_{\gb r\in\mathbb Z_+^\ell}^{} \lie a[\gb r]$ such that $\lie g:=\lie a[\gb 0]$ is a finite-dimensional semisimple Lie algebra over $\mathbb C$ and  $\lie a[\gb r]$ is a finite-dimensional $\lie g$-module for all $\gb r\in\mathbb Z_+^\ell$. Set $$\lie a_+=\opl_{\substack{\gb r\in\mathbb Z_+^\ell: \\ \deg(\gb r)\ne 0}}^{}\lie a[\gb r]$$
which is an ideal of $\lie a$. In particular, $\mathscr D(\lie a)$ is a $\mathbb Z_+$-graded Lie algebra with $\lie a[0]=\lie g$ and $\mathscr D(\lie a_+)$ is an ideal of $\mathscr D(\lie a)$. Moreover, $\lie a[r]$ is a finite-dimensional $\lie g$-module for all $r\in\mathbb Z_+$.

Fix a Cartan subalgebra $\lie h$ of $\lie g$ and denote by $R$ the associated root system. Fix also a choice of positive roots $R^+\subseteq R$ and let $\lie g=\lie n^-\oplus\lie h\oplus \lie n^+$ be the corresponding triangular decomposition. The associated simple roots and fundamental weights will be denoted by $\alpha_i$ and $\omega_i, i\in I$, where $I$ an indexing set for the nodes of the Dynkin diagram of $\lie g$. Denote by $P$ and $Q$ the weight and root lattices and let $P^+$ and $Q^+$ be the $\mathbb Z_+$-span of the fundamental weights and simple roots, respectively. For notational convenience, fix a Chevalley basis $\{x_\alpha^\pm, h_i:\alpha\in R^+, i\in I\}$ for $\lie g$.

\subsection{On finite-dimensional modules}

For any Lie algebra $\lie l$, let $\mathcal F(\lie l)$ be the category of finite-dimensional $\lie l$-modules. Given $\lambda\in P^+$, let $V(\lambda)$ be a simple object of $\mathcal F(\lie g)$ of highest-weight $\lambda$. Then, $V(\lambda)$ is isomorphic to the $\lie g$-module generated by a vector $v_\lambda$ satisfying the defining relations
\begin{equation*}
\lie n^+v_\lambda=0, \qquad hv_\lambda=\lambda(h)v_\lambda, \qquad (x_{\alpha_i}^-)^{\lambda(h_i)+1}v_\lambda=0 \qquad\text{for all } \qquad h\in\lie h, i\in I.
\end{equation*}
Denote by $V_\mu$ the weight space of $V$ of weight $\mu\in P$ and, for $V\in\mathcal F(\lie g)$, let
\begin{equation*}
\ch V = \sum_{\mu\in P} \dim V_\mu\ e^\mu
\end{equation*}
denote the formal character of $V$.

We will need the following elementary lemma.

\begin{lem}\label{l:hwvecs}
Let $V$ be a finite-dimensional $\lie g$-module and suppose $l\in\mathbb Z_{\ge 1}, \nu_k\in P, v_k\in V_{\nu_k}$, for $k=1,\dots,l$, are such that $V=\tsum_{k=1}^l U(\lie n^-)v_k$. Fix a decomposition $V= \opl_{j=1}^m V_j$ where $m\in\mathbb Z_{\ge 1},  V_j\cong V(\mu_j)$ for some $\mu_j\in P^+$, and let $\pi_j:V\to V_j$ be the associated projection for $j=1,\dots,m$. Then, there exist distinct $k_1,\dots,k_m\in\{1,\dots,l\}$ such that $\nu_{k_j}=\mu_j$ and $\pi_j(v_{k_j})\ne 0$.\hfill\qedsymbol
\end{lem}

\subsection{The main category}

Let $\mathcal G=\mathcal G(\lie a)$ be the category of $\mathbb Z^\ell$-graded $\lie a$-modules with finite-dimensional graded pieces. In particular, if $V\in\mathcal G$, then the graded piece $V[\gb r]$ is a finite-dimensional $\lie g$-module for every $\gb r\in\mathbb Z^\ell$. The morphisms in $\mathcal G$ are $\lie a$-module maps which are morphisms in $\cel V$. Let $\mathcal D = \mathcal G(\mathscr D(\lie a))$ be the category of $\mathbb Z$-graded $\mathscr D(\lie a)$-modules with finite-dimensional graded pieces.

Given $\gb r\in\mathbb Z^\ell$, denote by $\gb t^\gb r$ the monomial $t_1^{r_1}\cdots t_\ell^{r_\ell}\in\mathbb Z[[t_1^{\pm 1},\dots,t_\ell^{\pm 1}]]$ and, for $V\in\mathcal G$, define its graded character by
\begin{equation*}
\gch V = \sum_{\gb r\in\mathbb Z^\ell} \ch V[\gb r]\ \gb t^{\gb r}.
\end{equation*}
Consider the grade-shifting functor $\tau_{\gb r}:\mathcal G\to\mathcal G$, i.e., $\tau_{\gb r}V[\gb s]=V[\gb s+\gb r]$ for all $\gb s\in\mathbb Z^\ell$ and $\tau_{\gb r}$ is the identity on morphisms. Clearly, $\gch \tau_\gb r V=\gch V\ \gb t^\gb r$.
Similarly, for $r\in\mathbb Z$, one considers the functor $\tau_r:\cal G\to\cal G$.
We also have a functor $\ev:\mathcal F(\lie g)\to\mathcal G$ obtained by extending the action of $\lie g$ on a module $V$ to one of $\lie a$ on $V$ by setting $\lie a_+V=0$. We denote by $\ev_{\gb r}$ the composition $\tau_{\gb r}\ev:\mathcal F(\lie g)\to\mathcal G$. Similarly, we construct functors $\ev_r:\mathcal F(\lie g)\to\cal G$ for all $r\in\mathbb Z$. Set
\begin{equation*}
V(\lambda,\gb r)=\ev_{\gb r}V(\lambda) \qquad\text{and}\qquad V(\lambda,r)=\ev_rV(\lambda) \qquad\text{for all}\qquad \lambda\in P^+,\gb r\in\mathbb Z^\ell, r\in\mathbb Z.
\end{equation*}
We obviously have an isomorphism of $\mathscr D(\lie a)$-modules
\begin{equation*}
\mathscr D(V(\lambda,\gb r))\cong V(\lambda,\deg(\gb r)) \qquad\text{for all}\qquad \gb r\in\mathbb Z^\ell.
\end{equation*}
Set $\Lambda=P^+\times\mathbb Z^\ell$ and $\mathscr D(\Lambda)=\Lambda\times\mathbb Z$. Standard arguments prove the following theorem (cf. \cite[Proposition 2.3]{ckr:faces}).

\begin{thm}
If $V$ is a simple object of $\mathcal G$, then $V$ is isomorphic to $V(\lambda,\gb r)$ for a unique $(\lambda,\gb r)\in\Lambda$. In particular, if $V$ is a simple object of $\mathcal D$, then $V$ is isomorphic to $V(\lambda,r)$ for a unique $(\lambda, r)\in\mathscr D(\Lambda)$.\hfill\qedsymbol
\end{thm}

Given $V\in\mathcal G$, set
$$[V:V(\lambda,\gb r)]=\dim\Hom_\lie g(V[\gb r]:V(\lambda)) \qquad\text{for all}\qquad \lambda\in P^+, \gb r\in\mathbb Z^\ell$$
and notice that
\begin{equation}
\gch V = \sum_{(\lambda,\gb r)\in\Lambda} [V:V(\lambda,\gb r)]\ \ch V(\lambda)\ \gb t^{\gb r}.
\end{equation}

\subsection{Bounded objects and tensor products}

Given $\gb r\in\mathbb Z^\ell$, define the full subcategories $\mathcal G^\gb r$ and $\mathcal G_\gb r$ of $\mathcal G$ by the requirement
\begin{equation*}
V\in\mathcal G^\gb r,\  V[\gb s] \ne 0\ \Longrightarrow\ \gb r-\gb s\in\mathbb Z_+^\ell \qquad\text{and}\qquad V\in\mathcal G_\gb r,\  V[\gb s] \ne 0\ \Longrightarrow\ \gb s-\gb r\in\mathbb Z_+^\ell.
\end{equation*}
Then, define the subcategory $\mathcal G^b$ by the requirement
\begin{equation*}
V\in\mathcal G^b\ \Longleftrightarrow\ V\in\mathcal G^\gb r \qquad\text{for some}\qquad \gb r\in\mathbb Z^\ell.
\end{equation*}
Thus, the objects of $\mathcal G^b$ are those with grade bounded from above.
Similarly one defines the category $\mathcal G_b$ of objects with grade bounded from below.
Notice that if $V\in\mathcal G_b$, the module $\mathscr D(V)$ has grade bounded from below, but not necessarily finite-dimensional graded pieces.

The enveloping algebra $U(\lie a)$ is naturally an $\lie a$-module under $x\cdot a=[x,a]$ where $[x,a]=xa-ax$ is the commutator of $U(\lie a)$ and has a natural $\mathbb Z_+^\ell$-gradation. The following propositions are easily established (the proofs can be found in \cite{bia:thesis}).

\begin{prop}\label{p:gpUa}
For all $\gb r\in\mathbb Z_+^\ell$ we have an isomorphism of $\lie g$-modules $U(\lie a_+)[\gb r]\cong S^{(\gb r)}(\lie a)$ where
$$S^{(\gb r)}(\lie a) := \opl_{k}^{}\left( \otm_{\gb s\in\mathbb Z_+^\ell\setminus\{\gb 0\}}^{} {\rm Sym}^{k(\gb s)}\lie a[\gb s]\right)$$
and the sum is over all functions $k: \mathbb Z_+^\ell\setminus\{\gb 0\} \to \mathbb Z_+$ such that $\sum\limits_{\gb s\in\mathbb Z_+^\ell\setminus\{\gb 0\}} k(\gb s)\gb s=\gb r$. In particular, $U(\lie a_+)\in\mathcal G_b$.
\hfill\qedsymbol
\end{prop}

In the next proposition we equip $\wedge^j\lie a_+$ with the natural $\mathbb Z_+^\ell$-gradation coming from that of $\lie a$.

\begin{prop}\label{p:wedgea+}
For all $\gb r\in\mathbb Z_+^\ell$ and $j\in\mathbb Z_+$ we have an isomorphism of $\lie g$-modules
$$(\wedge^j\lie a_+)[\gb r]\cong \opl_{k}^{}\left( \otm_{\gb s\in\mathbb Z_+^\ell\setminus\{\gb 0\}}^{}\wedge^{k(\gb s)}\lie a[\gb s]\right),$$
where the sum is over all functions $k: \mathbb Z_+^\ell\setminus\{\gb 0\} \to \mathbb Z_+$ such that $\sum\limits_{\gb s\in\mathbb Z_+^\ell\setminus\{\gb 0\}} k(\gb s)\gb s=\gb r$ and $\sum\limits_{\gb s\in\mathbb Z_+^\ell\setminus\{\gb 0\}} k(\gb s)=j$. In particular, $\wedge^j\lie a_+\in\mathcal G_b$.
\hfill\qedsymbol
\end{prop}

Notice that the categories $\mathcal G_b$ and $\mathcal G^b$ become tensor categories by setting
\begin{equation*}
(V\otimes W)[\gb r]= \opl_{\substack{(\gb s,\gb s')\in\mathbb Z^\ell\times\mathbb Z^\ell:\\ \gb s+\gb s'=\gb r}}^{} V[\gb s]\otimes V[\gb s'].
\end{equation*}

We have the following corollary of Proposition \ref{p:gpUa}.

\begin{cor}\label{c:proj}
For all $V\in\mathcal G_b$, $U(\lie a)\otimes_{U(\lie g)} V$ is an object of $\mathcal G_b$.\hfill\qedsymbol
\end{cor}

\subsection{Projective modules}

Our first result is the construction of projective resolutions for the simple objects of $\mathcal G$.
Given $V\in\mathcal G_b$, define
\begin{equation*}
\mathscr P(V):= U(\lie a)\otimes_{U(\lie g)} V.
\end{equation*}
By Corollary \ref{c:proj}, $\mathscr P(V)\in\mathcal G_b$. Thus, the assignment $V\mapsto\mathscr P(V)$ can be naturally extended to an endo-functor $\mathscr P$ of $\mathcal G_b$.

Given $j\in\mathbb Z_+$, define
\begin{equation*}
\mathscr P_j(V): = \mathscr P((\wedge^j\lie a_+)\otimes V)
\end{equation*}
which is again an object of $\mathcal G_b$, since $\wedge^j\lie a_+\in\mathcal G_b$ by Proposition \ref{p:wedgea+}.  Notice that if $V=V[\gb r]$ for some $\gb r\in\mathbb Z^\ell$, then
\begin{equation*}
\mathscr P_j(V)[\gb s]=0 \qquad\text{for all}\qquad \gb s\in\mathbb Z^\ell \qquad\text{such that}\qquad \deg(\gb s)<j+\deg(\gb r).
\end{equation*}
Set $P_j(\lambda,\gb r)=\mathscr P_j(V(\lambda,\gb r)), \ P_j(\lambda,r)=\mathscr P_j(V(\lambda,r))$,  and notice that
\begin{equation*}
\mathscr D(P_j(\lambda,\gb r))\cong P_j(\lambda,\deg(\gb r)) \qquad\text{for all}\qquad j\in\mathbb Z_+, \lambda\in P^+, \gb r\in\mathbb Z^\ell.
\end{equation*}

Given $(\mu,\gb s), (\lambda,\gb r)$, we say that $(\mu,\gb s)$ covers $(\lambda,\gb r)$ if $\gb s-\gb r\in\mathbb Z_+^\ell\setminus\{\gb 0\}$ and $\mu-\lambda$ is a weight of the $\lie g$-module $\lie a[\gb s-\gb r]$. Further, consider the partial order $\preccurlyeq$ obtained by the transitive and reflexive closure of the relation $(\lambda,\gb r)\prec (\mu,\gb s)$ if $(\mu,\gb s)$ covers $(\lambda,\gb r)$. Similarly one defines a partial order $\preccurlyeq$ on $\mathscr D(\Lambda)$.

For $\ell=1$, the next proposition was proved in \cite{cg:hwcat,ckr:faces} and all the statements can be proved in the same manner for arbitrary $\ell$ (all details can be found in \cite{bia:thesis}).

\begin{prop}\label{p:proj}
Let $\lambda\in P^+,\gb r,\gb s\in\mathbb Z^\ell$, and $V\in\mathcal G$.
\begin{enumerate}[(a)]
\item $\mathscr P(V)$ is a projective object of $\mathcal G$ and canonically surjects onto $V$.
\item\label{p:projrel} $P(\lambda,\gb r)$ is isomorphic to the $\lie a$-module generated by a vector $v$ of degree $\gb r$ satisfying the following defining relations
\begin{equation*}
\lie n^+v=0, \qquad hv=\lambda(h)v, \qquad (x_{\alpha_i}^-)^{\lambda(h_i)+1}v=0 \qquad\forall\ h\in\lie h, i\in I.
\end{equation*}
\item Let $K(\lambda,\gb r)$ be the kernel of the projection $P(\lambda,\gb r)\to V(\lambda,\gb r)$. Then, $[K(\lambda,\gb r):V(\mu,\gb s)]\ne 0$ only if ${\rm Hom}_\lie g(\lie a[\gb s - \gb r ] \otimes V(\lambda),v(\mu))\ne 0$. In particular, $(\mu,\gb s)$ covers $(\lambda,\gb r)$.
\item $[V:V(\lambda,\gb r)]=\dim{\rm Hom}_{\mathcal G}(P(\lambda,\gb r), V)$ \ and \ ${\rm Hom}_{\mathcal G}(P(\lambda,\gb r), V)\cong {\rm Hom}_{\lie g}(V(\lambda),V[\gb r])$.
\item For $j>0$, $[P_j(\lambda,\gb r):V(\mu,\gb s)]\ne 0$ only if $(\lambda,\gb r)\prec (\mu,\gb s)$.
\item If $\mathscr D(V)=V[r]$ for some $r\in\mathbb Z$, then $\mathscr P(V)$ is the projective cover of $V$ in $\mathcal G$. In particular, $P(\lambda,\gb r)$ is the projective cover of $V(\lambda,\gb r)$.
\hfill\qedsymbol
\end{enumerate}
\end{prop}

In what follows we use that we have an isomorphism of $\lie g$-modules
\begin{equation*}
\mathscr P(V) \cong S(\lie a_+)\otimes V \qquad\text{for all}\qquad V\in\mathcal G_b.
\end{equation*}
We also fix $(\lambda,\gb r)\in\Lambda$.
Let $d_0:P(\lambda,\gb r)\to V(\lambda,\gb r)$ be given by $d_0(u\otimes v) = uv$ for all $u\in\lie a_+, v\in V(\lambda,\gb r)$ and, for $j>0$, let $d_j:P_j(\lambda,\gb r)\to P_{j-1}(\lambda,\gb r)$ be the $\lie a$-module map determined by
\begin{equation*}
d_j = D_j\otimes {\rm Id}_{V(\lambda,\gb r)},
\end{equation*}
where $D=(D_j)_{j>0}$ is the Koszul differential on the Chevalley-Eilenberg complex for $\lie a_+$.

\begin{prop}
The sequence
$$\cdots{\longrightarrow} P_2(\lambda,\gb r)\stackrel{d_2}{\longrightarrow} P_1(\lambda,\gb r) \stackrel{d_1}{\longrightarrow} P(\lambda,\gb r)\stackrel{d_0}{\longrightarrow} V(\lambda,\gb r)\to 0$$
is a projective resolution of $V(\lambda,\gb r)$ in $\mathcal G_b$ and its image in $\cal D$ under $\mathscr D$ is a projective resolution of $V(\lambda,\deg(\gb r))$.
\end{prop}

\begin{proof}
Since $\mathscr D$ is the identity when regarded as an endo-functor on the category of vector spaces, it suffices to prove one of the two statements. The second statement was proved in \cite[Proposition 2.7(ii)]{ckr:faces}.
\end{proof}

\subsection{Restricting the class of simple modules}

Next we study certain full subcategories of $\cal G$ obtained by choosing the simple objects which belong to it. This will eventually allow us to define a certain generalization of Kirillov-Reshetikhin modules (see Section \ref{ss:genrel}). Thus, given any subset $\Gamma$ of $\Lambda$, let $\mathcal G[\Gamma]$ be the full subcategory of $\mathcal G$ consisting of objects $V$ such that
\begin{equation*}
[V:V(\lambda,\gb r)]\ne 0\implies (\lambda,\gb r)\in\Gamma.
\end{equation*}
Quite clearly, $\mathcal G=\mathcal G[\Lambda]$ and the isomorphism classes of simple objects in $\mathcal G[\Gamma]$ are indexed by the elements of $\Gamma$. In particular, if $\Gamma$ is finite, then $\mathcal G[\Gamma]\subseteq\mathcal F(\lie a)$.

For $V\in\mathcal G$ and $\Gamma\subseteq\Lambda$, define
\begin{equation*}
V_\Gamma^+ = \{v\in V[\gb r]_\lambda: (\lambda,\gb r)\in\Gamma, \lie n^+v=0\}, \qquad V_\Gamma=U(\lie g)V_\Gamma^+, \qquad\text{and}\qquad V^\Gamma=V/V_{\Lambda\setminus\Gamma}.
\end{equation*}
It is easy to see that $V_\Gamma$ and $V^\Gamma$ are $\lie g$-modules which can be naturally regarded as objects in the category $\cel V$. Moreover,
$$\Hom_\lie g(V_\Gamma[\gb r],V(\lambda))\ne 0 \qquad\text{only if}\qquad  (\lambda,\gb r)\in\Gamma$$
and similarly for $V^\Gamma$.
Also, note that if $f\in \Hom_{\cal G}(V,W)$, then $f(V_\Gamma^+)\subset W_\Gamma^+$ and, hence, the restriction $f_\Gamma$ of $f$ to $V_\Gamma$ is an element of $\Hom_\lie g(V_\Gamma, W_\Gamma)$. Moreover, since $f(V_{\Lambda\setminus\Gamma})\subset W_{\Lambda\setminus\Gamma}$, we have a homomorphism of $\lie g$-modules $f^\Gamma: V^\Gamma \to W^\Gamma$. In general, it is not true that either $V_\Gamma$ or $V^\Gamma$ are $\lie a$-modules, but we shall see that this is true under certain conditions (see Lemma \ref{l:inGGamma} below). In that case, $f_\Gamma$ and $f^\Gamma$ are clearly morphisms in ${\cal G}[\Gamma]$.

The next proposition is a more general version of parts of \cite[Proposition 3.4]{ckr:faces}.

\begin{prop}\label{vGamma} Let $V \in \cal G$ be such that $V^\Gamma \in {\cal G}[\Gamma]$.

\begin{enumerate}[(a)]
\item \label{vGammaa} For all $(\nu, \gb t) \in \Gamma$, $[V^{\Gamma} : V(\nu,\gb t)] = [V : V(\nu,\gb t)]$.
\item \label{vGammab} If $V$ is projective in $\cal G$, then $V^{\Gamma}$ is projective in ${\cal G}[\Gamma]$.
\item \label{vGammac} If $W \in {\cal G}[\Gamma]$, then $\Hom_{\cal G}(V, W) \cong \Hom_{{\cal G}[\Gamma]}(V^{\Gamma}, W).$
\end{enumerate}
\end{prop}

\begin{proof}
For \eqref{vGammaa}, given $\gb s\in \mathbb Z^\ell$, consider the short exact sequence of graded $\lie g$-modules
$$0 \to V_{\Lambda \setminus \Gamma}[\gb s] \to V[\gb s] \to V^{\Gamma}[\gb s] \to 0.$$
Since each finite-dimensional $\lie g$-module is semisimple, we have
$$\dim \Hom_{\lie g}(V(\mu), V[\gb s]) = \dim \Hom_{\lie g}(V(\mu), V^{\Gamma}[\gb s]) + \dim \Hom_{\lie g}(V(\mu), V_{\Lambda \setminus \Gamma}[\gb s]).$$
Since, by definition of $V_{\Lambda \setminus \Gamma}$, we also have $\Hom_{\lie g}(V(\mu), V_{\Lambda \setminus \Gamma}[\gb s]) = 0$ for all $(\mu, \gb s) \in \Gamma$, part (a)  follows.

For proving \eqref{vGammab}, let $U,W \in {\cal G}[\Gamma]$ and consider a homomorphism $g: V^\Gamma \to U$ and an epimorphism $h: W\to U$. Since $V$ is projective, there exists a homomorphism $f : V \to W$ in $\cal G$ such that $hf = g\pi_{V}$, where $\pi_{V} : V \to V^{\Gamma}$ is the canonical projection. Furthermore, since $W \in {\cal G}[\Gamma]$, we have $f(V_{\Lambda \setminus \Gamma}) = 0$ which implies that there exists a unique homomorphism $\overline{f} : V^{\Gamma} \to W$ in ${\cal G}[\Gamma]$ such that $\overline{f}\pi_{V} = f$. Therefore, $h \overline{f} \pi_{V} = h f =  g \pi_{V}$ and the surjectivity of $\pi_{V}$  implies that $h \overline{f} = g$ as desired.

Finally,  given $f \in \Hom_{{\cal G}}(V, W)$, by hypothesis on $W$ we have $f(V_{\Lambda \setminus \Gamma}) = 0$ and, hence, there exists a unique homomorphism $\psi_f : V^{\Gamma} \to W$ in ${\cal G}[\Gamma]$ such that $\psi_f\pi_{V} = f$. Now, it is routine to show that
$$\psi : \Hom_{\cal G}(V, W) \to \Hom_{{\cal G}(\Gamma)}(V^{\Gamma}, W) \qquad \text{ given by} \qquad f \mapsto \psi_f$$
defines a linear isomorphism.
\end{proof}

A subset $\Gamma$ of $\Lambda$ is said to be convex  with respect to $\preccurlyeq$ if
$$(\lambda,\gb r) \preccurlyeq (\nu,\gb p) \preccurlyeq (\mu,\gb s) \qquad \text{ and }\qquad  (\lambda,\gb r), (\mu,\gb s) \in \Gamma \implies (\nu,\gb p)\in \Gamma.$$
We note that a convex set was called ``an interval closed set'' in \cite{ckr:faces}.
For $\ell=1$, the next two results were proved in \cite[Propositions 3.2 -- 3.4]{ckr:faces} and the proofs for arbitrary $\ell$ are analogous (all details can be found in \cite{bia:thesis}).

\begin{lem}\label{l:inGGamma}
Let $\Gamma\subseteq\Lambda$ and $V \in \mathcal G$.
\begin{enumerate}[(a)]
\item Given $(\lambda,\gb r), (\mu,\gb s)\in\Lambda$, $\Ext^1_{\mathcal G}(V(\lambda,\gb r),V(\mu,\gb s)) \ne0$ only if $(\mu,\gb s)$ covers $(\lambda,\gb r)$.
\item  If $V_\Gamma$ is not an $\lie a$-submodule of $V$, then there exist $(\mu,\gb s)\in\Lambda(V)\setminus\Gamma$ and $(\lambda,\gb r)\in\Lambda(V)\cap\Gamma$ such that $(\mu,\gb s)$ covers $(\lambda,\gb r)$.
\item  Suppose that $\Gamma$ is finite and convex with respect to $\preccurlyeq$.
\begin{enumerate}
    \item Assume that for any $(\lambda,\gb r)\in \Lambda(V)\setminus \Gamma$ there exists $(\mu,\gb s)\in \Gamma$ with $(\lambda,\gb r) \preccurlyeq (\mu,\gb s)$. Then, $V_\Gamma\in \mathcal G[\Gamma]$. Furthermore if $U$ is an $\lie a$-submodule of $V$, then $U_\Gamma, (V/U)_\Gamma \in \mathcal G[\Gamma]$ and $(V/U)_\Gamma \cong V_\Gamma/U_\Gamma$.
    \item Assume that for any $(\lambda,\gb r)\in \Lambda(V)\setminus \Gamma$ there exists $(\mu,\gb s)\in \Gamma$ with $(\mu,\gb s) \preccurlyeq (\lambda,\gb r)$. Then, $V^\Gamma\in \mathcal G[\Gamma]$. Furthermore, if $U$ is an $\lie a$-submodule of $V$, then $U^\Gamma, (V/U)^\Gamma \in \mathcal G[\Gamma]$ and $(V/U)^\Gamma \cong V^\Gamma/U^\Gamma$.\hfill\qedsymbol
\end{enumerate}
\end{enumerate}
\end{lem}

\begin{prop}
Suppose $\Gamma\subseteq\Lambda$ is finite and convex and let $(\lambda,\gb r),(\mu,\gb s)\in\Gamma$.
\begin{enumerate}[(a)]
\item $P_j(\lambda,\gb r)^\Gamma\in\mathcal G[\Gamma]$ for all $j\in\mathbb Z_+$.
\item $P(\lambda,\gb r)^\Gamma$ is the projective cover of $V(\lambda,\gb r)$ in $\mathcal G[\Gamma]$.
\item $[P(\lambda,\gb r):V(\mu,\gb s)]=[P(\lambda,\gb r)^\Gamma:V(\mu,\gb s)]=\dim \Hom_{\mathcal G[\Gamma]}(P(\mu,\gb s)^\Gamma,P(\lambda,\gb r)^\Gamma)$.
\item $\Hom_{\mathcal G}(P(\mu,\gb s),P(\lambda,\gb r))\cong\Hom_{\mathcal G[\Gamma]}(P(\mu,\gb s)^\Gamma,P(\lambda,\gb r)^\Gamma)$.
\item The induced sequence
$$\cdots{\longrightarrow} P_2(\lambda,\gb r)^\Gamma\stackrel{d_2^\Gamma}{\longrightarrow} P_1(\lambda,\gb r)^\Gamma \stackrel{d_1^\Gamma}{\longrightarrow} P(\lambda,\gb r)^\Gamma\stackrel{d_0^\Gamma}{\longrightarrow} V(\lambda,\gb r)\to 0$$
is a finite projective resolution of $V(\lambda,\gb r)$ in $\mathcal G[\Gamma]$.\hfill\qedsymbol
\end{enumerate}
\end{prop}

\section{A Generalization of Restricted KR modules}

The main goal of this section is to compute the graded character of the modules $P(\lambda,\gb r)^\Gamma$ under certain conditions on $\lie a$ and $\Gamma$. To achieve this, we compute the space of extensions between any two simple objects in $\mathcal G$. Finally, we give a presentation of $P(\lambda,\gb r)^\Gamma$ in terms of generators and relations which allows us to regard these modules as generalizations of restricted Kirillov-Reshetikhin modules (in the sense of \cite{cm:kr,cm:krg}).

\subsection{Extensions}

For the next few subsections we assume $\lie a[\gb r]=0$ if  $\deg(\gb r)>1$ and set $V_j=\lie a[\gb e_j]$ for $j=1,\dots,\ell$, and  $V=\opl_{j=1}^\ell V_j$. Notice that $\lie a \cong \lie g\ltimes V$. In this case, the statements of Propositions \ref{p:gpUa} and \ref{p:wedgea+} simplify and we get

\begin{equation}
U(\lie a_+)[\gb r]\cong_\lie g {\rm Sym}^{r_1}V_1\otimes\cdots\otimes {\rm Sym}^{r_\ell}V_\ell
\end{equation}
and
\begin{equation}\label{e:wedgedeg1}
(\wedge^j\lie a_+)[\gb r]\cong_\lie g
\begin{cases}
\wedge^{r_1}V_1\otimes\cdots\otimes \wedge^{r_\ell}V_\ell,& \text{if } \deg(\gb r)=j,\\
0,& \text{otherwise.}
\end{cases}
\end{equation}

For $\ell=1$, the next theorem was proved in \cite[Propositions 6.1 and 6.2]{ckr:faces} and the proof for arbitrary $\ell$ is analogous
(all the details can be found in \cite{bia:thesis}).

\begin{thm}\label{t:exts}
Let $(\lambda,\gb r), (\mu,\gb s)\in\Lambda, r=\deg(\gb r), s=\deg(\gb s)$, and $j\ge 0$. Then,
\begin{equation*}
\Ext^j_{\mathcal G}(V(\lambda,\gb r), V(\mu,\gb s)) \cong
\begin{cases}
\Hom_\lie g((\wedge^j\lie a_+)[\gb s-\gb r]\otimes V(\lambda),V(\mu)), &\text{ if } j=s-r,\\
0, &\text{ otherwise.}
\end{cases}
\end{equation*}
Moreover, if $\Gamma$ is finite and convex and $(\lambda,\gb r), (\mu,\gb s)\in\Gamma$, then
$$\Ext^j_{\mathcal G[\Gamma]}(V(\lambda,\gb r), V(\mu,\gb s))\cong \Ext^j_{\mathcal G}(V(\lambda,\gb r), V(\mu,\gb s)).$$\hfill\qedsymbol
\end{thm}

\begin{rem}
For the proof of the last statement of the above theorem in the case $\ell=1$ \cite[Proposition 4.2]{ckr:faces}, the authors of \cite{ckr:faces} refer to the proof of \cite[Proposition 3.3]{cg:koszul} in order to show that the canonical map
$$\Hom_{\cal G}(P_j(\lambda,\gb r), V(\mu,\gb s)) \to \Hom_{{\cal G}[\Gamma]}(P_j(\lambda,\gb r)^\Gamma, V(\mu,\gb s))$$
is an isomorphism for all $j\in\mathbb Z_+$. This clearly follows from Proposition \ref{vGamma}\eqref{vGammac}.
It is worth remarking that the last statement of Theorem \ref{t:exts} (and its proof) remains valid without the assumption $\lie a[\gb r]=0$ if $\deg(\gb r)>1$. For partial information on the first statement of Theorem \ref{t:exts} (with $\ell=1$)  without this assumption, see \cite{fer:msc}.
\end{rem}

\subsection{A character formula}

From now on we assume that $\wt(V)\ne\{0\}$. We also fix a subset $\Psi\subseteq\wt(V)$ satisfying
\begin{gather}\notag
\sum_{\nu\in\Psi} m_\nu\nu= \sum_{\mu\in\wt(V)} n_\mu\mu\ \ (m_\nu,n_\mu\in\mathbb Z_+) \qquad  \Longrightarrow\qquad  \sum_{\nu\in\Psi} m_\nu\le \sum_{\mu\in\wt(V)} n_\mu\\ \label{e:polytope} \text{and}\\\notag
\sum_{\nu\in\Psi} m_\nu= \sum_{\mu\in\wt(V)}n_\mu \qquad\text{only if}\qquad n_\mu=0 \qquad\text{for all}\qquad \mu\notin\Psi.
\end{gather}

\begin{rem}
It is proved in \cite{kharid:polytopes}  that such subsets are precisely those lying on a proper face of the convex polytope determined by $\wt(V)$. A generalized version of this condition was considered in \cite{fer:msc} for studying algebras $\lie a$ with nonzero graded parts at any degree (with $\ell=1$).
\end{rem}

Consider the reflexive and transitive binary relation on $P$ given by
\begin{equation*}
\mu\le_\Psi\lambda \qquad\text{if}\qquad \lambda-\mu\in\mathbb Z_+\Psi,
\end{equation*}
where $\mathbb Z_+\Psi$ is the $\mathbb Z_+$-span of $\Psi$. Set also
\begin{equation*}
d_\Psi(\mu,\lambda)=\min\left\{\sum_{\nu\in\Psi}m_\nu: \lambda-\mu=\sum_{\nu\in\Psi}m_\nu\nu, m_\nu\in\mathbb Z_+ \ \forall\ \nu\in\Psi \right\}.
\end{equation*}
By \cite[Proposition 5.2]{ckr:faces}, $\le_\Psi$ is in fact a partial order on $P$ and
\begin{equation*}
d_\Psi(\nu,\mu)+d_\Psi(\mu,\lambda) = d_\Psi(\nu,\lambda) \qquad\text{whenever}\qquad \nu\le_\Psi\mu\le_\Psi\lambda.
\end{equation*}
Moreover, it induces a refinement of the partial order $\preccurlyeq$ on $\Lambda$ by setting
\begin{equation*}
(\lambda,\gb r)\preccurlyeq_\Psi (\mu,\gb s) \qquad\text{if}\qquad \lambda\le_\Psi\mu,\qquad \gb s-\gb r\in\mathbb Z_+^\ell, \qquad\text{and}\qquad d_\Psi(\lambda,\mu)=\deg(\gb s-\gb r).
\end{equation*}
The next proposition was proved in \cite[Lemma 5.5]{ckr:faces} for $\ell=1$ and the proof extends naturally to the present context (as usual, the details can be found in \cite{bia:thesis}).

\begin{prop}
Let $\Gamma\subseteq\Lambda$ be finite and convex with respect to $\preccurlyeq_\Psi$ and assume that there exists $(\lambda,\gb r)\in\Lambda$ such that $(\lambda,\gb r)\preccurlyeq_\Psi(\mu,\gb s)$ for all $(\mu,\gb s)\in\Gamma$. Then:
\begin{enumerate}[(a)]
\item $\Hom_{\mathcal G[\Gamma]}(P(\mu,\gb s)^\Gamma,P(\nu,\gb t)^\Gamma)\ne 0$ only if $(\mu,\gb s)\preccurlyeq_\Psi(\nu,\gb t)$.
\item $\Ext^j_{\mathcal G[\Gamma]}(V(\mu,\gb s),V(\nu,\gb t))\ne 0$ only if $(\nu,\gb t) \preccurlyeq_\Psi (\mu,\gb s) $ and $j=d_\Psi(\nu,\mu)$.
\hfill\qedsymbol
\end{enumerate}
\end{prop}

Henceforth, assume that $\Gamma\subseteq\Lambda$ is finite and convex with respect to $\preccurlyeq_\Psi$. Also, fix an enumeration of $\Gamma$ compatible with $\preccurlyeq_\Psi$.

Given $\gb r\in\mathbb Z^\ell$, set $(-\gb t)^{\gb r}:=\prod_{j=1}^\ell (-t_j)^{r_j} = (-1)^{\deg(\gb r)}\gb t^{\gb r}\in\mathbb Z[[t_1^{\pm 1}, \dots, t_\ell^{\pm 1}]]$. Consider the matrices
\begin{gather*}
A(\gb t) = \left([P(\lambda,\gb r):V(\mu,\gb s)]\ \gb t^{\gb s-\gb r}\right)_{(\mu,\gb s),(\lambda,\gb r)\in\Gamma}\\\notag \text{and}\\
E(\gb t) = \left(\dim\Ext^{\deg(\gb s-\gb r)}_{\mathcal G}(V(\lambda,\gb r),V(\mu,\gb s))\ \gb t^{\gb s-\gb r}\right)_{(\mu,\gb s),(\lambda,\gb r)\in\Gamma}.
\end{gather*}
If $\gb s-\gb r\notin\mathbb Z_+^\ell$, then the entry $((\mu,\gb s),(\lambda,\gb r))$ of both matrices vanish and, hence, both matrices are lower triangular.
The proof of the next lemma is analogous to that contained in the proof of \cite[Theorem 5.6]{ckr:faces} (extra details already in the present context can be found in \cite{bia:thesis}).

\begin{lem} The above defined matrices  satisfy $A(\gb t)E(-\gb t) ={\rm Id}$.\hfill\qedsymbol
\end{lem}

Using \eqref{e:wedgedeg1} in place of its $\ell=1$ version, one proves the next theorem in the same manner that it was proved in \cite[Theorem 2]{cg:minp}.

\begin{thm}\label{t:ac}
For all $\lambda\in P^+$ such that $(\lambda,\gb n)\in \Gamma$ we have
$$\ch(V(\lambda))\ (-\gb t)^{\gb n} = \sum_{(\mu,\gb r)\in\Gamma} (-1)^{\deg(\gb r)}\ c^{\lambda,\gb n}_{\mu,\gb r}\ \gch P(\mu,\gb r)^\Gamma,$$
where
$$c^{\lambda,\gb n}_{\mu,\gb r} = \dim \Ext^j_{\mathcal G}(V(\lambda,\gb n), V(\mu,\gb r)) =\dim\Hom_\lie g((\otm_{i=1}^\ell \wedge^{r_i-n_i} V_i)\otimes V(\lambda), V(\mu)),$$
$\gb r = (r_1,\cdots,r_\ell)$, $\gb n = (n_1,\cdots,n_\ell)$ and $j=\deg \gb r - \deg \gb n$.\hfill\qedsymbol
\end{thm}

\begin{rem}
In \cite[Theorem 2]{cg:minp} the above theorem is stated to be true for all $\lambda\in P^+$ instead of for all $(\lambda,\gb n)\in\Lambda$. It is easy to see that the present hypothesis is used in the proof and, in fact, needed. The character formula given by this theorem is of recursive nature (see Example \ref{ex:char} below).
\end{rem}

\subsection{Generator and relations}\label{ss:genrel}

In addition to \eqref{e:polytope} we now suppose that $\Psi$ also satisfies
\begin{gather}\label{e:Psiextra}
\Psi\cap P^+ = \emptyset, \quad  (\wt(V)+\mathbb Z_+\Psi)\cap P^+ \quad \text{is finite}, \quad \text{and}\quad  \xi+\alpha_i\notin \Psi \quad \text{for all}\quad  \xi\in\wt(V)\cap P^+, i\in I.
\end{gather}

\begin{exe}
Let $V$ be the adjoint representation of $\lie g$ and, hence, $\wt(V)= R\cup\{0\}$. Given $\mu\in P^+\setminus\{0\}$, consider
\begin{equation*}
\Psi(\mu) = \{\alpha\in R\mid (\alpha,\mu) = \min\limits_{\beta\in R} (\beta,\mu)\}\subseteq -R^+
\end{equation*}
It was shown in \cite[Lemma 2.3]{cg:koszul} that $\Psi(\mu)$ satisfy \eqref{e:polytope} and it clearly satisfies \eqref{e:Psiextra} since it is contained in $-R^+$. The relevance of the subsets of type $\Psi(\mu)$ will be explained below (see also  \cite[Section 3.6]{cg:minp}).\hfill$\diamond$
\end{exe}

Given $\lambda\in P^+, \gb r\in\mathbb Z_+^\ell, r\in\mathbb Z_+$, define
\begin{equation*}
\Gamma_\Psi(\lambda,\gb r)=\{(\mu,\gb s)\in\Lambda:(\lambda,\gb r)\preccurlyeq_\Psi(\mu,\gb s)\}.
\end{equation*}
Henceforth we fix  $(\lambda,\gb r)\in\Lambda$ and set $\Gamma=\Gamma_\Psi(\lambda,\gb r)$.

\begin{lem} \label{l:gammapsi} Let $(\mu,\gb s)\in\Gamma$.
\begin{enumerate}[(a)]
\item If $(\mu,\gb s')\in\Gamma$, then $\deg(\gb s')=\deg(\gb s)$.
\item The set $\Gamma$ is finite and convex with respect to $\preccurlyeq_\Psi$. In particular, $P(\mu,\gb s)^\Gamma$ is a finite-dimensional object of $\mathcal G[\Gamma]$.
\item If $(\nu,\gb s')\in\Gamma$ is such that $(\mu,\gb s)\preccurlyeq_\Psi (\nu,\gb s')$, then $(\nu,\gb s'-\gb s)\in \Gamma_\Psi(\mu,\gb 0)$.
\end{enumerate}
\end{lem}

\begin{proof}
Part (a) is a consequence of condition \eqref{e:polytope} (cf. \cite[Proposition 2.2(i)]{cg:minp}).
It is clear from the construction of $\Gamma$ that it is convex. Since $(\lambda,\gb r)\preccurlyeq_\Psi(\nu,\gb t)$ implies $\nu\in\lambda+\mathbb Z_+\Psi$, the second condition in \eqref{e:Psiextra} together with part (a) implies that the set $\{(\nu,\deg(\gb t)): (\nu,\gb t)\in\Gamma\}$ is finite. On the other hand, $(\lambda,\gb r)\preccurlyeq_\Psi(\nu,\gb t)$ also implies that $\gb t-\gb r\in\mathbb Z_+^\ell$ and, hence, $\Gamma$ is finite.
The second statement in part (b) then follows from Proposition \ref{l:inGGamma}.
Part (c) is proved as in the $\ell=1$ case \cite[Proposition 2.2(iii)]{cg:minp}.
\end{proof}

We now prove the following generalization of \cite[Theorem 1]{cg:minp}.

\begin{thm}\label{t:pgenrel}
For all $(\mu,\gb s)\in\Gamma$, the module $P(\mu,\gb s)^\Gamma$ is isomorphic to the $\lie a$-module $M(\mu,\gb s)$ generated by an element $w$ of degree $\gb s$ satisfying the defining relations
\begin{gather*}
\lie n^+w=0, \qquad hw=\mu(h)w, \qquad (x_{\alpha_i}^-)^{\mu(h_i)+1}w=0, \qquad\text{for all}\qquad h\in\lie h, i\in I,\\
xw=0 \qquad\text{for all}\qquad x\in V_\xi, \ \xi\in\wt(V)\setminus \Psi.
\end{gather*}
\end{thm}

\begin{proof}
Let $v$ be a generator of $P(\mu,\gb s)$ as in Proposition \ref{p:proj}\eqref{p:projrel} and keep denoting by $v$ its image in $P(\mu,\gb s)^\Gamma$ which is nonzero since $(\mu,\gb s)\in\Gamma$. We begin by showing that $v$ satisfies the above relations for $w$. It evidently satisfies those in the first line. Suppose there exists $\xi\in\wt(V)\setminus\Psi$ and $x\in V_\xi$ such that $xv\ne 0$ and choose $\xi$ maximal (with respect to the usual partial order on $P$) with this property. The first condition in \eqref{e:Psiextra} and the maximality of $\xi$ imply that $\xi\in P^+$. The third condition in \eqref{e:Psiextra} implies that $\lie n^+xv=0$. Indeed, given $i\in I$, we have $x_{\alpha_i}^+xv=[x_{\alpha_i}^+,x]v \in V_{\xi+\alpha_i}$. Now, since $\xi+\alpha_i\notin \Psi$, the maximality of $\xi$ implies $[x_{\alpha_i}^+,x]v=0$ as claimed.
It follows that, $\mu+\xi\in P^+$ and, since $P(\mu,\gb s)^\Gamma\in\mathcal G[\Gamma]$, we conclude that that $(\mu+\xi,\gb s+\gb e_i)\in\Gamma$ and $(\mu,\gb s)\prec_\Psi(\mu+\xi,\gb s+\gb e_i)$. By the previous lemma, we then have $(\mu+\xi,\gb e_i)\in\Gamma_\Psi(\mu,\gb 0)$ which implies $(\mu,\gb 0)\preccurlyeq_\Psi (\mu+\xi,\gb e_i)$. Thus, by definition of $\preccurlyeq_\Psi$, we have $\xi\in\mathbb Z_+\Psi$ and $\xi=\sum_\nu m_\nu \nu$, where $\nu \in \Psi$, $m_\nu\in \mathbb Z_+$ and $\sum_\nu m_\nu = \deg(\gb e_i)=1$. Hence, $\xi \in \Psi$, yielding the desired contradiction.

It is clear from the definition of $M(\mu,\gb s)$ that it is in $\mathcal G$. Moreover, by the previous paragraph, we have a short exact sequence $0\to K\to M(\mu,\gb s)\to P(\mu,\gb s)^\Gamma\to 0$. Since $P(\mu,\gb s)^\Gamma$ is projective in $\mathcal G[\Gamma]$ and $M(\mu,\gb s)$ is clearly indecomposable,
it suffices to show that this is a sequence in $\mathcal G[\Gamma]$, i.e., that $M(\mu,\gb s)\in\mathcal G[\Gamma]$. It is also clear from the definition of $M(\mu,\gb s)$ that $V(\mu,\gb s)$ is its unique irreducible quotient. Since $\tau_\gb sM(\mu,\gb 0)$ is obviously isomorphic to $M(\mu,\gb s)$, the last part of the previous lemma implies that it suffices to show that $M(\mu,\gb 0)\in\mathcal G[\Gamma_\Psi(\mu,\gb 0)]$. So, for the remainder of the proof, assume that $\Gamma=\Gamma(\mu,\gb 0)$ and, to shorten notation, write $M=M(\mu,\gb 0)$. Thus, we want to show that
\begin{equation}\label{e:pgenrel}
[M:V(\nu,\gb k)]\ne 0\qquad\text{only if}\qquad (\mu,\gb 0)\preccurlyeq_\Psi(\nu,\gb k).
\end{equation}

Choose a basis of $\lie a$ consisting of a Chevalley basis for $\lie g$ and bases for the weight spaces of $V_j, j=1,\dots,\ell$. Set $U(\Psi)$ to be the set of PBW monomials (in some order) formed by basis elements from $\opl_{\nu\in\Psi}^{} V_\nu$. A standard application of the PBW theorem implies that
\begin{equation*}
M = U(\lie n^-)U(\Psi) w.
\end{equation*}
This together with Lemma \ref{l:hwvecs} clearly implies \eqref{e:pgenrel}.
\end{proof}

\begin{cor} \label{c:ac1} \
\begin{enumerate}[(a)]
\item \label{c:1gch} Let $(\mu,\gb s), (\nu,\gb s') \in \Gamma$ such that $(\mu,\gb s) \preccurlyeq_\Psi (\nu,\gb s')$. Then,
$\gch P(\nu,\gb s')^\Gamma =  \gch P(\nu,\gb s'-\gb s)^{\Gamma_\psi(\mu,\gb0)}\ \gb t^\gb s.$
\item \label{c:2gch}  For all $\lambda\in P^+$ such that $(\lambda,\gb 0)\in \Gamma$ we have
$\ch(V(\lambda)) = \tsum_{(\mu,\gb r)\in\Gamma}^{} (-1)^{\deg(\gb r)}\ c^{\lambda,\gb 0}_{\mu,\gb r}\ \gch P(\mu,\gb 0)^{\Gamma_\Psi(\mu,\gb 0)}.$
\end{enumerate}
\end{cor}

\proof Part (a) is an immediate consequence of Theorem \ref{t:pgenrel} and Lemma \ref{l:gammapsi}(c). Part (b) follows from (a) and Theorem \ref{t:ac}.
\endproof

\subsection{Multigraded Kirillov-Reshetikhin Modules}

Theorem \ref{t:pgenrel} allows us to regard the modules $P(\lambda,\gb r)^\Gamma$ as generalizations of the  graded classical limits of Kirillov-Reshetikhin modules. Indeed, consider the particular case that $\ell=1$ and $V=\lie a_+$ is the adjoint representation of $\lie g$. For $\lie g$ of classical type, $\lambda=m\omega_i$ for some $i\in I,m\in\mathbb Z_+$, and $\Psi=\Psi(\mu)=\Psi(\omega_i)$ or $\Psi=\emptyset$ (depending on $i$), it was shown in \cite{cm:kr,cm:krg} that the classical limit of the Kirillov-Reshetikhin modules can be described in terms of generator and relations which are ``exactly'' those given in the Theorem \ref{t:pgenrel}. Because of this we say that $P(\lambda,\gb r)^\Gamma$ can be regarded as generalizations of the classical limits of Kirillov-Reshetikhin modules. However, the original graded Kirilov-Reshetikhin modules are modules for the algebra $\lie a=\lie g\otimes\mathbb C[t]$ which does not satisfy the hypothesis $\lie a[r]=0$ for $r>1$ and this is why we have quotation marks on the word ``exactly'' above. In the next theorem, we show that, if $\lie a=\lie g\otimes A$ with $\lie g$ of classical type and $A=\mathbb C[t_1,\cdots,t_\ell]$, then the $\lie a$-module defined by such relations factors to a module for a certain quotient $\lie b$ of $\lie a$ and, as a $\lie b$-module, it is isomorphic to $P_\lie b(\lambda,\gb r)^\Gamma$ where $P_\lie b(\lambda,\gb r)^\Gamma$ denotes the $\lie b$-module constructed similarly to the $\lie a$-module $P(\lambda,\gb r)^\Gamma$.

Notice that with the present hypothesis on $\lie a$, we have
\begin{equation*}
\wt(\lie a_+) = \wt(\lie g) = R\cup\{0\}.
\end{equation*}
We shall need the following lemma (see \cite[Section 1]{cha:fer} for details) where the hypothesis that $\lie g$ is of classical type is crucial.

\begin{lem} \label{l:rootssum} Let $\beta = \sum_{i\in I} m_{\beta,i} \alpha_i \in R^+$ \ $(m_{\beta,i}\in \mathbb Z_+)$.
\begin{enumerate}[(a)]
	\item \label{l:rootssum.a} If  $m_{\beta,j} = 2$ for some $j\in I$, then there exist $\beta_1,\beta_2 \in R^+$ such that $\beta = \beta_1 + \beta_2$ and $m_{\beta_1,j}=m_{\beta_2,j}=1$.
	\item \label{l:rootssum.b} If  $\beta$ is not simple and $m_{\beta,j} = 1$ for some $j\in I$, then there exist $\beta_0,\beta_1 \in R^+$ such that $\beta = \beta_0 + \beta_1$ and $m_{\beta_0,j}=0$ and $m_{\beta_1,j}=1$. \hfill \qedsymbol
\end{enumerate}
\end{lem}

For $i\in I$, set
\begin{equation*}
\Psi_i = \left\{-\alpha: \alpha=\sum_{j\in I}m_j\alpha_j\in R^+, m_i=2\right\}.
\end{equation*}
It follows that either $\Psi_i=\emptyset$ or $\Psi_i=\Psi(\omega_i)$. Assume that $|I|=n$ and that $I$ is labeled by the set $\{1,\dots,n\}$ in the usual way. Then, given $\lambda\in P^+$, set
\begin{equation*}
i_\lambda = \max\{1, i: \lambda(h_i)\ne 0, i \text{ is not  a spin node}\} \qquad\text{and}\qquad \Psi_\lambda = \Psi_{i_\lambda}.
\end{equation*}
Also, for $\gb r\in\mathbb Z^\ell$, let
$N(\lambda,\gb r)$ be the $\lie a$-module generated by a vector $v$ of grade $\gb r$ satisfying the defining relations
\begin{gather}
\lie n^+v=0, \qquad hv=\lambda(h)v, \qquad (x_{\alpha_j}^-)^{\lambda(h_j)+1}v=0, \qquad\text{for all}\qquad h\in\lie h, j\in I, \notag\\ \text{and}\label{e:rel2} \\
xv=0 \qquad\text{for all}\qquad x\in (\lie a_+)_\xi, \ \ \xi\in \wt(\lie a_+)\setminus \Psi_\lambda \notag.
\end{gather}
Given $m\ge 0$, set
$$A_m=\opl_{\substack{\gb r\in\mathbb Z_+^\ell: \\ \deg(\gb r)\ge m}}^{}A[\gb r] \qquad\text{and}\qquad
\lie a_m = \lie g\otimes A_m.$$

\begin{thm}\label{t:pgenrel2}
Let $\lambda\in P^+, \gb r\in\mathbb Z^\ell, \Psi=\Psi_\lambda, \Gamma = \Gamma_{\Psi}(\lambda,\gb r)$, and $\lie b = \lie a / \lie a_2$.
\begin{enumerate}[(a)]
\item $\lie a_2\ N(\lambda,\gb r)=0$. In particular, $N(\lambda,\gb r)$ factors to a $\lie b$-module.
\item  $N(\lambda,\gb r)$ is isomorphic to $P_\lie b(\lambda,\gb r)^\Gamma$.
\item If $\lambda=m\omega_i$ for some $m\in\mathbb Z_+, i\in I$, then $N(\lambda,\gb r)$ is isomorphic to the $\lie a$-module generated by a vector $v$ of
grade $\gb r$ satisfying the defining relations
\begin{gather*}
\lie n^+v=0, \qquad hv=\lambda(h)v, \qquad (x_{\alpha_j}^-)^{\lambda(h_j)+1}v=0, \qquad\text{for all}\qquad h\in\lie h, j\in I,\\
((\lie n^+\oplus\lie h)\otimes A_1)v=0 \qquad\text{and}\qquad (x_{\alpha_i}^-\otimes a)v=0, \qquad\text{for all}\qquad a\in A[1].
\end{gather*}
\end{enumerate}
\end{thm}

\begin{proof} In order to prove part (a), we shall prove by induction on $\text{ht} \ \eta$ that
\begin{equation} \label{Ind1}
	(x_\alpha^\pm \otimes a) \ N(\lambda,\gb r)_{\lambda-\eta}=0 \qquad  \text{ for all } \qquad   \alpha\in R^+ \  \text{ and } \  a\in A_2
\end{equation}
and
\begin{equation} \label{Ind2}
	(h \otimes b) \ N(\lambda,\gb r)_{\lambda-\eta}=0 \qquad \text{ for all } \qquad h\in \lie h \ \text{ and } \ b\in A_2.
\end{equation}
It clearly suffices to prove \eqref{Ind1} and \eqref{Ind2} assuming that $a$ and $b$ are monomials.

Assume ${\text{ht}} \ \eta = 0$. If $\alpha\in R^+\setminus -\Psi_\lambda$, \eqref{Ind1} follows immediately from the second line in \eqref{e:rel2}. Similarly, since $0\in\wt(\lie a_+)\setminus\Psi_\lambda$, we also have $(\lie h\otimes A_1)\ v = 0$. Finally, suppose $\alpha \in -\Psi_\lambda$ and write $a=cd$ with $c,d\in A_1$. By Lemma \ref{l:rootssum}\eqref{l:rootssum.a}, there exist $\beta_1,\beta_2\in R^+\setminus -\Psi_\lambda$ such that $\alpha = \beta_1+\beta_2$. Hence, $x^\pm_\alpha \otimes a \ v  =  [x^\pm_{\beta_1}\otimes c, x^\pm_{\beta_2}\otimes d]  \ v  =0$. This proves that induction begins.

Let $v' = (x_{\gamma_1}^- \otimes a_1) \cdots (x_{\gamma_s}^- \otimes a_s) \ v \in N(\lambda,\gb r)_{\lambda-\eta}$ with $\eta \ne 0$, $a_j\in A\setminus \{0\}$ and $\gamma_j \in R^+$ for $j\in \{1,\cdots,s\}$. Then, for all $x\in \lie a_2$ we have
$$ x \ v' = [x, x^-_{\gamma_1}\otimes a_1] (x^-_{\gamma_2} \otimes a_2) \cdots (x^-_{\gamma_s}\otimes a_s) \ v \ \ + \ \  (x^-_{\gamma_1} \otimes a_1) x (x^-_{\gamma_2} \otimes a_2) \cdots (x^-_{\gamma_s}\otimes a_s) \ v.$$
The second summand vanishes by induction hypothesis. Similarly, since $[x, x^-_{\gamma_1}\otimes a_1] \in \lie a_2$, the first summand vanishes and part (a) is proved.

For proving (b), let $N_\lie b(\lambda,\gb r)$ be the $\lie b$-module generated by a vector satisfying the defining relations \eqref{e:rel2} (with elements from $\lie a$ replaced by their images in $\lie b$). Since $\lie b[\gb s]=0$ whenever $\deg(\gb s)\ge 2$, it follows from Theorem \ref{t:pgenrel} that $N_\lie b(\lambda,\gb r)\cong P_\lie b(\lambda,\gb r)^\Gamma$. Moreover, it follows from part (a) that $N(\lambda,\gb r)$ is a quotient of $N_\lie b(\lambda,\gb r)$. For the converse, let $\mathcal I$ be the two-sided ideal of $U(\lie a)$ generated by $\lie a_2$. Then,
$U(\lie b)\cong U(\lie a)/\mathcal I.$ Let also $\mathcal N_\lie a$ be the left ideal of $U(\lie a)$ corresponding to the defining relations \eqref{e:rel2} and define $\mathcal N_\lie b$ similarly. Notice that $\mathcal I\subseteq \mathcal N_\lie a$ and that $\mathcal N_\lie b=\mathcal N_\lie a/\mathcal I$. It follows that we have isomorphisms of vector spaces
\begin{equation*}
N(\lambda,\gb r)\cong U(\lie a)/\mathcal N_\lie a \cong   (U(\lie a)/\mathcal I)/(\mathcal N_\lie a/\mathcal I) \cong  U(\lie b)/\mathcal N_\lie b \cong N_\lie b(\lambda,\gb r)\cong P_\lie b(\lambda,\gb r)^\Gamma.
\end{equation*}

Finally, we prove part (c). Let $N(\lambda,\gb r)'$ be the module defined by the relations in part (c). Comparing these relations with those in \eqref{e:rel2} it is immediate that $N(\lambda,\gb r)$ is a quotient of $N(\lambda,\gb r)'$. Conversely, to see that $N(\lambda,\gb r)'$ is a quotient of $N(\lambda,\gb r)$, we have to prove that $xv=0$ for all $x\in (\lie a_1)_\xi$, $\xi\in \wt(\lie a_1)\setminus \Psi_\lambda$. Evidently, it suffices to prove that
\begin{equation}\label{claim(c)}
	x_\alpha^- \otimes a \ v=0 \quad \text{ for all } \quad \alpha \in R^+\setminus -\Psi_\lambda \quad \text{ and } \quad a\in A_1.
\end{equation}
As before, it suffices to prove  \eqref{claim(c)} with $a$ being a monomial. Notice also that, for all $\alpha \in R^+$ and monomial $b\in A[n]$, we have:
\begin{equation}\label{step:degreeA}
(\lie h\otimes A_+) \ v=0 \quad \text{ and }\quad (x^-_\alpha\otimes b) \ v = 0 \implies (x^-_\alpha\otimes b') \ v = 0 \quad \text{ for all }\quad b'\in A_m \quad \text{with} \quad m>n.
\end{equation}

Now we prove \eqref{claim(c)}. Let $\alpha=\sum_j m_{\alpha,j} \alpha_j \in R^+\setminus -\Psi_\lambda$. Since $\lambda=m\omega_i$, $x^-_{\alpha_j} \ v =0$ for all $j\ne i$ and one easily deduces by induction on $\text{ht}\ \alpha$ that $x^-_{\alpha} v =0$ whenever $m_{\alpha,i}=0$. Using \eqref{step:degreeA} we then get \eqref{claim(c)} for all such $\alpha$. If $m_{\alpha,i}=1$, we proceed by induction on $\text{ht}\ \alpha$ as well. If $\text{ht}\ \alpha=1$, then $\alpha=\alpha_i$ and we are done by definition of $N(\lambda,\gb r)'$ together with \eqref{step:degreeA}. Let $\alpha\in R^+\setminus -\Psi_\lambda$ with $\text{ht} \ \alpha >1$. Then, by Lemma \ref{l:rootssum}\eqref{l:rootssum.b} , there exists a decomposition $\alpha = \beta_0+\beta_1$ where $\beta_0, \beta_1 \in R^+$ with $m_{\beta_0,i}=0$ and $m_{\beta_1,i}=1$. In particular, $x^-_{\beta_0,i}v=0$. Since $\text{ht}\ \beta_1 <\text{ht}\ \alpha$, the induction hypothesis implies that $(x^-_{\beta_1,i}\otimes a)v=0$ and, hence
$$(x^-_{\alpha}\otimes a) v = [x^-_{\beta_0,i}, x^-_{\beta_1,i}\otimes a ] \ v = 0.$$
Since, by definition of $\Psi_\lambda$, we have $m_{\alpha,i}\le 1$ for all $\alpha\in R^+\setminus-\Psi_\lambda$, this completes the proof.
\end{proof}

\begin{rem}
Theorem \ref{t:pgenrel2} remains valid when $\lie g$ is of exceptional type with restrictions on the support of $\lambda$ once the appropriate definition of $\Psi_\lambda$ is given. For instance, for $\lie g$ of type $E_6$, the result holds as long as $\lambda(h_{i_0})=0$ where $i_0$ is the trivalent node (cf. \cite{mp:e6}).
\end{rem}

In the case $A=\mathbb C[t]$, the graded characters of the modules $N(\lambda,0)$ with $\lambda=m\omega_i$ for some $i\in I,m\in\mathbb Z_+$, were studied in \cite{cha:fer,cm:kr} where it was proved that they coincide with certain graded characters for Kirillov-Reshetikhin modules conjectured in \cite{jap:rem} (where the grading comes from quantum considerations instead). If $\lie g$ is of exceptional type, one can take the module defined in part (c) of Theorem \ref{t:pgenrel2} as definition of $N(m\omega_i,\gb r)$. However, parts (a) and (b) do not hold for all $i\in I$. For $\lie g$ of type $G_2$ and $A=\mathbb C[t]$, the graded character of $N(m\omega_i,0)$ was studied in \cite{cm:krg}. For more general $\lambda$, it was conjectured in \cite{mou:reslim} (see also \cite{chakle}) that the module $N(\lambda, 0)$ is isomorphic to the graded classical limits of the associated minimal affinizations of quantum groups. Some particular cases of this conjecture were proved in \cite{cg:minp,mou:reslim,mp:e6}. The method used in \cite{cm:kr,cm:krg,mou:reslim,mp:e6} to study the graded characters of $N(\lambda, 0)$ becomes unmanageable as either $\lambda$ or $\ell$ grow. However, the formula in Corollary \ref{c:ac1}\eqref{c:2gch} can be used for computing $\gch N(\lambda,\gb 0)$ in a recursive manner as long as we are under the hypotheses of Theorem \ref{t:pgenrel2}. We now give an example of this procedure.

\begin{ex}\label{ex:char}
Let $\lie g$ be of type $D_n$ and $A=\mathbb C[t_1,\dots,t_\ell]$. If either $i=1$ or a spin node, then $\Psi_i=\emptyset$ and one easily sees that
$\Gamma_{\Psi_i}(m\omega_i,\gb 0)=\{(m\omega_i,\gb 0)\}$ and, hence,
$$N(m\omega_i,\gb 0)\cong V(m\omega_i,\gb 0).$$
If $i=2$, then
$$\Gamma_{\Psi_2}(m\omega_2,\gb 0)=\{((m-r)\omega_2,\gb r): 0\le r\le m, \gb r\in\mathbb Z_+^\ell, \deg(\gb r)=r\}.$$
Using Corollary \ref{c:ac1}\eqref{c:2gch}, we get
\begin{equation*}
\gch N(m\omega_2,\gb 0) = \sum_{\substack{\gb r\in\mathbb Z_+^\ell,\\ \deg(\gb r)\le m}} \ch V((m-\deg(\gb r))\omega_2)\ \gb t^{\gb r}.
\end{equation*}
We will omit the details. Notice that isotypical summands of $N(m\omega_i,\gb 0)$ when regarded as a $\lie g$-module do not depend on $\ell$ in the examples above. More precisely, given a finite-dimensional $\lie g$-module $V$, denote by $c_V(\mu)$ the multiplicity of $V(\mu)$ as a simple summand of $V$. Setting $V=N(\lambda,\gb 0)$ and $c_\ell(\lambda,\mu) = c_V(\mu)$, we have seen that, if $\lambda=m\omega_i$ for $i$ as above, then
\begin{equation}\label{e:chinvell}
c_\ell(\lambda,\mu) \ne 0 \qquad\Leftrightarrow\qquad c_1(\lambda,\mu)\ne 0.
\end{equation}
We now show that this is not true for all $i\in I$. Thus, let $i=3$. It was shown in \cite{cha:fer} that, for $\ell=1$, we have
\begin{equation*}
\gch N(m\omega_3,0) = \sum_{r=0}^m \ch V((m-r)\omega_3+r\omega_1)\ t^r.
\end{equation*}
We will show that for $\ell >1$ we have
\begin{gather}\notag
\gch N(2\omega_3,\gb 0) = \ch(V(2\omega_3)) + \sum_{j=1}^\ell \ch(V(\omega_3+\omega_1))\ t_j  + \sum_{1\le j< k\le \ell} \ch(V(\omega_2))\  t_jt_k   \\ \label{e:2omega3D}\\\notag  + \sum_{\substack{\gb r \in \mathbb Z_+^\ell \\ \deg(\gb r)=2} }^{}  \ch(V(2\omega_1))\ \gb t^{\gb r} + \sum_{1\le i < j < k \le \ell } \ \ch(V(0)) \  t_{i} t_{j} t_{k}.
\end{gather}
Notice that $\ch(V(\omega_2))$ appears in \eqref{e:2omega3D} iff $\ell\ge 2$ while $\ch(V(\omega_0))$ appears iff $\ell\ge 3$.

In order to prove \eqref{e:2omega3D}, begin observing that
$$\Gamma_{\Psi_3}(2\omega_3,\gb 0) = \{(2\omega_3,\gb 0),\ (\omega_3+\omega_1,\gb r),\ (\omega_2,\gb r'), \ (2\omega_1,\gb r'), \ (0,\gb r''):  \deg(\gb r)=1, \deg(\gb r')=2, \deg(\gb r'') = 3\}.$$
Setting $\Gamma = \Gamma_{\Psi_3}(2\omega_3,\gb 0)$, Theorem \ref{t:ac} implies that
\begin{align*}
\ch(V(2\omega_3))  = & c_{2\omega_3,\gb 0}^{2\omega_3,\gb 0} \ \gch P(2\omega_3,\gb 0)^{\Gamma }-\sum_{j=1}^\ell c_{\omega_3+\omega_1,\gb e_j}^{2\omega_3,\gb 0} \ \gch P(\omega_3+\omega_1,\gb e_j)^{\Gamma }+\sum_{\substack{\gb r \in \mathbb Z_+^\ell \\ \deg(\gb r)=2}}^{} c_{\omega_2,\gb s}^{2\omega_3,\gb 0} \ \gch P(\omega_2,\gb r)^{\Gamma } + \\
&+ \sum_{\substack{\gb r \in \mathbb Z_+^\ell \\ \deg(\gb r)=2}}^{}  \ \gch P(2\omega_1,\gb r)^{\Gamma } - \sum_{\substack{\gb r \in \mathbb Z_+^\ell \\ \deg(\gb r)=3}}^{} c_{0,\gb r}^{2\omega_3,\gb 0} \ \gch P(0,\gb r)^{\Gamma}.
\end{align*}
A straightforward computation using the formulas for decompositions of tensor products of finite-dimensional $\lie g$-modules shows that
\begin{gather*}
c_{2\omega_3,\gb 0}^{2\omega_3,\gb 0}=  c_{\omega_3+\omega_1,\gb e_j}^{2\omega_3,\gb 0}= c_{\omega_2,\gb e_j+\gb e_k}^{2\omega_3,\gb 0}= c_{\omega_2,2\gb e_j}^{2\omega_3,\gb 0}= c_{2\omega_1,\gb e_j+\gb e_k}^{2\omega_3,\gb 0}= c_{0,\gb r}^{2\omega_3,\gb 0} = 1, \qquad j\ne k,\ \ \gb r\in\mathbb Z_+^\ell, \deg(\gb r)=3\\\text{and}\qquad c_{2\omega_1,2\gb e_j}^{2\omega_3,\gb 0} = 0 \qquad\text{for all}\qquad j=1,\dots,\ell.
\end{gather*}
Hence,
\begin{align*}
\gch N(2\omega_3,\gb 0)  &= \ch(V(2\omega_3))+\sum_{j=1}^\ell \gch P(\omega_3+\omega_1,\gb e_j)^{\Gamma}-\sum_{\substack{\gb r \in \mathbb Z_+^\ell \\ \deg(\gb r)=2}}^{} \gch P(\omega_2,\gb r)^{\Gamma  } \\
&- \sum_{1\le j< k\le\ell}^{} \gch P(2\omega_1,\gb e_j+\gb e_k)^{\Gamma } + \sum_{\substack{\gb r \in \mathbb Z_+^\ell \\ \deg(\gb r)=3}}^{}  \ \gch P(0,\gb r)^{\Gamma}.
\end{align*}
By Corollary \ref{c:ac1}\eqref{c:1gch} and Theorem \ref{t:pgenrel2}(b) we have
$$\gch P(\omega_3+\omega_1,\gb e_j)^{\Gamma} = \gch N(\omega_3+\omega_1,\gb 0)\ t_j, \qquad \gch P(\omega_2,\gb r)^{\Gamma} = \gch N(\omega_2,\gb 0)\ \gb t^{\gb r},$$
$$\gch P(2\omega_1,\gb e_j+\gb e_k)^{\Gamma} =  \gch N(2\omega_1,\gb 0 )\ t_jt_k \qquad \text{and} \qquad \gch P(0,\gb r)^{\Gamma}  = \gch N(0,\gb 0)\ \gb t^{\gb r}. $$
Using Theorem \ref{t:ac} for computing each of these terms as before and so on, one eventually obtains \eqref{e:2omega3D}. We omit the details.  \hfill$\diamond$
\end{ex}

\begin{rem}\label{r:depell}
It is interesting to remark that the graded local Weyl modules do satisfy \eqref{e:chinvell} for all $\lambda,\mu,\ell$ with the graded local Weyl module in place of $N(\lambda,\gb 0)$ in the definition of $c_\ell(\lambda,\mu)$ (see \cite{cfk}).
\end{rem}

\end{document}